\documentclass[letterpaper, 10 pt, conference]{ieeeconf}  

\IEEEoverridecommandlockouts                              
\usepackage{amsmath}
\usepackage{amsmath,amssymb,amsfonts} 
\usepackage{graphicx}
\usepackage{algorithm}
\usepackage{algpseudocode}
\usepackage{hyperref}
\hypersetup{hidelinks=true}
\usepackage{textcomp}
\usepackage{epstopdf}
\usepackage{booktabs} 
\usepackage{hyperref}
\usepackage{amsfonts}       
\usepackage{multicol}
\usepackage{multirow} 
\usepackage{subcaption}
\usepackage{xcolor}

\usepackage{amsthm}

\newcommand{\R}{\mathbb{R}}

\newcommand{\bmat}[1]{\begin{bmatrix}#1\end{bmatrix}}

\newcommand{\mcl}[1]{\mathcal{#1}}
\newcommand{\mbf}[1]{\mathbf{#1}}

\newcommand{\N}{\mathbb{N}}

\theoremstyle{plain}

\newtheorem{thm}{Theorem}
\newtheorem{defn}[thm]{Definition}

\title{\LARGE \bf
Picard Iteration for Parameter Estimation in Nonlinear Dynamic Models of Aircraft and Spacecraft}

\author{Aleksandr Talitckii$^{1}$ and Matthew M. Peet$^{1}$
\thanks{This works was supported by the National Science Foundation under grants NSF CCF-2323532}
\thanks{$^{1}$Aleksandr Talitckii and Matthew M. Peet are with the School for the Engineering of Matter, Transport and Energy, Arizona State University, Tempe, AZ, 85298, USA. {\tt\small atalitck@asu.edu} and {\tt\small mpeet@asu.edu }}
}

\begin{document}

\maketitle
\thispagestyle{empty}
\pagestyle{empty}

\begin{abstract}
The attitude dynamics of aircraft and spacecraft exhibit significantly nonlinear behaviour. In spacecraft, torque is generated through reaction wheels and control moment gyros. In aircraft, torque is generated using lift on control surfaces. In both cases, complex geometries, unique configurations, and internal/environmental changes imply that models must be identified, verified, and updated using in-flight experimental data. However, this data is often noisy, sparsely sampled, and partial in that modeled states may not be directly measurable. In this paper, we propose a method for estimating key parameters in realistic Ordinary Differential Equation (ODE) models of both spacecraft and aircraft dynamics. This method avoids the need to directly measure state derivatives by coupling sampled outputs using the Picard mapping -- an integral constraint on the solution of the parameterized ODE. This constraint is then enforced, and optimal parameter estimates are found using a gradient contraction algorithm. This algorithm is applied to well-studied models of spacecraft and aircraft motion. First, the algorithm is used to estimate the inertia tensor in a 4 control-moment gyro (CMG) model of spacecraft motion. Second, we estimate the 28 higher-order control surface coefficients in a model of the F/A-18 aircraft.      
%
\end{abstract}

\section{Introduction}  

The fundamental laws of rigid body motion allow us to construct accurate models for the attitude dynamics of aircraft and spacecraft using nonlinear Ordinary Differential Equations (ODEs). However, due to complex geometries, unique configurations, and dependence on environmental factors, certain model parameters may need to be determined using experimental data. Furthermore, this data may be sparsely sampled at irregular intervals, noisy, and include only partial measurements of internal states.


The problem of parameter estimation is typically posed as grey-box system identification. While many convex methods exist for grey-box identification of linear systems in both the time and frequency domain (e.g.~\cite{ljung2010perspectives}), for nonlinear systems with sparse sampling, noise, and partial measurements, the problem becomes more challenging. Specifically, this parameter estimation problem may be formulated as a minimization of some loss function as applied to an output of the solution map, $\phi(t,x_0,\theta)$ -- which defines the dependence of the ODE solution on the initial conditions $x_0$ and parameters $\theta$. 


Unfortunately, for nonlinear ODEs, an analytic expression for the solution map is rarely available. The class of ``shooting methods'' avoids this problem by applying numerical simulation for given initial conditions and parameter values. These numerical solutions can then be used to approximate gradients of the loss function iteratively in gradient descent~\cite{byrd1988approximate, more2006levenberg}. However, such approaches require a large number of simulations at each iteration with no guarantee that the resulting gradients are accurate. Furthermore, when time horizons are long or dynamics are stiff~\cite{zhao2015accurate} (typical for attitude dynamics), numerical simulations are highly sensitive to initial condition and parameters -- implying that gradient approximations will be unreliable.





In this paper, we avoid numerical problems associated with shooting methods in parameter estimation by formulating the associated optimization problem directly in terms of an infinite-dimensional variable, $\mbf u(t)$, which represents the solution associated with a given choice of parameters, $\theta$, and initial conditions, $x_0$, while enforcing the constraint\vspace{-2mm}
\begin{equation}
\mbf u(t) = x_0 + \int\nolimits_0^t f(s, \mbf u(s), \theta)ds,\label{eqn:constraint}\vspace{-2mm}
\end{equation} 
This approach obviates the need for an analytic solution map, while avoiding numerical problems associated with shooting methods. This constraint can be represented as $\mbf u=\mcl P \mbf u$ where $\mcl P$ is the Picard mapping, where we recall that on sufficiently short time intervals, $\mcl P$ is a contraction and for any $\mbf u_0$, $\mbf u=\lim_{k\rightarrow \infty}\mcl P^k \mbf u_0$ satisfies Eqn.~\eqref{eqn:constraint} -- i.e. $\mbf u=\mcl P \mbf u$. Furthermore, for any $\mbf u_k$, the gradients $\nabla_\theta \mcl P \mbf u_k$ and $\nabla_{x_0} \mcl P \mbf u_k$ can be computed analytically.




Roughly speaking, then, to solve the parameter estimation problem, we use a gradient contractive algorithm, alternating between a gradient step and a Picard mapping, so that $\theta_{k+1}=\theta_k+\alpha \nabla_\theta L(\mcl P_{\theta_k} \mbf u_k)$ and $\mbf u_{k+1}=\mcl P_{\theta_{k+1}}\mbf u_k$ where $L$ is an objective function and $\alpha$ is a gradient descent step-size. Convergence then implies that the corresponding solution $\mbf u^*$ associated with $\theta^*$ will satisfy the dynamics of the model with parameters $\theta^*$. Of course, this requires an extension of the Picard mapping to indefinite time intervals and certain other technical elements which we omit for brevity. For a complete exposition of this algorithmic approach along with convergence properties in a more general context, we refer to~\cite{talitckii2024picard}. An overview of the algorithm is provided in Sec.~\ref{sec:algorithm}.


In this paper, we examine how this type of gradient contractive parameter estimation algorithm can be applied to the problem of identifying critical parameters in the attitude dynamics of stiff ODE models of aircraft and spacecraft. 

In the case of spacecraft dynamics, we estimate the inertia tensor in the attitude dynamics of spacecraft with $4$ variable-speed single gimbal control moment gyroscopes (SGCMGs), using the dynamics proposed in~\cite{schaub1998feedback}. In the case of aircraft, we identify control surface coefficients in a proposed model of the F/A-18 aircraft~\cite{napolitano1996estimation}, which is known to have local instabilities in the form of falling leaf mode.


The paper is organized as follows. In Sec.~\ref{sec:problem_statement}, we propose an alternative formulation of the parameter estimation task. Then, in Sec.~\ref{sec:algorithm}, we present algorithms for solving the parameter estimation problem. Next, we apply the proposed algorithm to a model of spacecraft attitude dynamics using control moment gyros in Sec.~\ref{sec:control-moment-gyros}. Finally, in Sec.~\ref{sec:FA18-model}, we estimate parameters for a model of the F/A-18 aircraft.

\subsection{Notation}
We denote by $\N$, $\R$, and $\R_{+}$ the natural numbers, the real numbers, and the non-negative real numbers. We denote by
$\mbf I$ and $\mbf 0$ the identity and zero matrix. For $J \in \N$ we denote $\overline{1, J} = \{j \in \N \;|\; 1 \leq j\leq J\}$. For $a, b \in \R_+$ we denote ${\lfloor a/b \rfloor} :=\underset{m\in\N, a \geq mb}{\texttt{argmax}} m$ and $a \bmod{b} := a- \lfloor a/b\rfloor b$. For $a, b \in \R^3$, we use $a \times b$ to denote the cross product of vectors $a$ and $b$. For compact $S \subset \R^n$, we denote $C^k(S)$ to be the space of $k$ times continuously differentiable functions with norm $\|f\| = \sup_{s\in S} \|f(s)\|_2$. 
For $u \in C^1(S)$, we denote by $\nabla_x u$ the gradient $\nabla_x u  = \bmat{\frac{\partial}{\partial x_1}u, \dots, \frac{\partial}{\partial x_m}u }^T$.  
For compact $X \subseteq \R^{n}$ we denote $\Pi_X:\R^{n} \rightarrow X$ to be the projection operator so that $\Pi_X(y) = \texttt{argmin}_{x \in X} \|x- y\|_2$. 

 
\section{Parameter Estimation Problem for Nonlinear Differential Equations}\label{sec:problem_statement}

In this section, we formulate the parameter estimation problem for nonlinear ODEs as an optimization task. Specifically, we consider a differential equation model in the state-space representation of the form  \vspace{-2mm} 
\begin{align}
\dot {\mbf x}(t) &= f(t, {\mbf x}(t), \theta) \qquad {\mbf x}(0)  = x \notag\\[-1mm]
y(t) & = g(t, {\mbf x}(t), \theta), \label{eqn:ODE} \\[-7mm] \notag
\end{align}
where $t\in \Gamma \subset \R$ is time, ${\mbf x}(t) \in \R^{n_x}$ are system states, $y(t) \in  \R^{n_y}$ are measured outputs, $\theta \in \Theta$ are unknown parameters and $x \in X$ are initial conditions. We assume that sets $\Theta \subset \R^{n_\theta}$ and $X \subset \R^{n_x}$ are compact and convex. The vector field $f:\R\times\R^{n_x}\times\Theta \rightarrow \R^{n_x}$ and the output function $g:\R\times\R^{n_x}\times\Theta \rightarrow \R^{n_y}$ are continuously differentiable.  In addition, we assume the existence and uniqueness of the solution of Differential Equation~\eqref{eqn:ODE}, where the solution map, $\phi:\Gamma\times X\times \Theta \rightarrow \R^{n_x}$, satisfies  \vspace{-2mm}
    \begin{equation}
        \partial_t \phi(t, x, \theta) = f(t, \phi(t, x, \theta), \theta), \quad
        \phi( 0, x, \theta ) = x. \label{eqn:solutionmap} \vspace{-2mm}
    \end{equation}

Our goal is to identify unknown parameters $\theta \in \Theta$ and initial conditions $x \in X$ based on a set of measurements of the dynamical system of the form\vspace{-2mm}
\begin{equation*}
    Y\hspace{-1mm} =\hspace{-1mm}\{ (y_i, \hspace{-0.5mm}t_i) \;\hspace{-1mm} | \hspace{-1mm}\; y_i \hspace{-0.5mm}=\hspace{-0.5mm} (1+n_i)g(t_i, \phi(t_i, x, \theta), \theta) + m_i,\; \hspace{-1mm}
 i \in \overline{1, N_s}  \},\vspace{-2mm}
\end{equation*}
where $m_i\in \R^{n_y}$ and $n_i \in \R$ represent measurement errors.

The parameter estimation task can be formulated as an optimization problem of the form\vspace{-2mm}
\begin{equation}
    \min_{\substack{x \in X,  \theta \in \Theta}} L(x, \theta, \phi(\cdot, x, \theta)), \label{eqn:optimization_problem_unconstrained}\vspace{-4mm}
\end{equation}
\begin{equation*}
\text{where }L(x, \theta, \phi(\cdot, x, \theta))\hspace{-1mm} =\hspace{-1mm} \frac{1}{2N_s}\sum_{i=1}^{N_s}\hspace{-1mm} \|y_i - g(t_i, \phi(t_i, x, \theta), \theta)\|^2\hspace{-1mm}. \vspace{-2mm}
\end{equation*} 
The objective function in Optimization~\eqref{eqn:optimization_problem_unconstrained} depends on the solution map. However, for nonlinear systems, the solution map is typically unknown. We propose the constrained formulation of the parameter estimation task of the form\vspace{-2mm}
\begin{align} \label{eqn:optimization_problem_constrained}
   \min_{\substack{x \in X, \theta \in \Theta, {\mbf u} \in C(\Gamma)}} \hspace{-4mm} &\; L(x, \theta, {\mbf u})\\[-3.5mm] 
     \text{s.t.} &\;  {\mbf u}(t) = x+\int_{0}^t f(s,{\mbf u}(s),\theta)ds, \notag  \\[-8mm]\notag
\end{align}
where we introduced the infinite-dimensional decision variable ${\mbf u} \in C(\R)$. The integral constraint in~\eqref{eqn:optimization_problem_constrained} implies that ${\mbf u}(t)$ satisfies the ODE in~\eqref{eqn:ODE} with initial condition ${\mbf u}(0)=x$ and parameter values $\theta$.  The variable $\mbf u$ allows us to avoid the need for explicit knowledge of the solution map, but now requires solving an infinite-dimensional optimization problem.



Optimality conditions for the optimization problem in~\eqref{eqn:optimization_problem_constrained} can be derived using methods from~\cite{luenberger1997optimization}. Specifically, the KKT conditions for~\eqref{eqn:optimization_problem_constrained} have the form\vspace{-2mm}
\begin{equation} 
\nabla_{x, \theta} {L}(x, \theta, \mbf u) + \sum_{i = 1}^{N_s}\nabla_{u_i} L(x, \theta, \mbf u) \nabla_{x, \theta} \phi(t_i, x, \theta) = 0,
\label{eqn:optimality_conditions}\vspace{-4mm}
\end{equation}
where \vspace{-2mm}
\begin{equation*} 
\nabla_{u_i} L(x, \theta, \mbf u)\hspace{-1mm}:=\hspace{-1mm}\frac{1}{N_s}\hspace{-1mm} (y_i - g(t_i, \mbf u(t_i), \theta))^T \nabla_u g(t_i, \mbf u(t_i), \theta).\vspace{-2mm}
\end{equation*} 
The introduced notation $\nabla_{u_i} L(x, \theta, \mbf u)$ represents the gradient of the loss function with respect to the estimated solution at the sampled time points $t_i$ --  i.e., $\nabla_{u_i} L := \nabla_{u(t_i)} L$.  Note that the KKT optimality conditions in~\eqref{eqn:optimality_conditions} are identical to the optimality conditions for the unconstrained formulation in~\eqref{eqn:optimization_problem_unconstrained}. Therefore, the optimization problems in~\eqref{eqn:optimization_problem_unconstrained} and~\eqref{eqn:optimization_problem_constrained} are equivalent.

\section{Picard Iterations for Parameter Estimation}
\label{sec:algorithm}

In this section, we propose an algorithm for solving the optimization problem in Eqn.~\eqref{eqn:optimization_problem_constrained}. We accomplish this task in three steps. First, we reformulate the optimization problem in~\eqref{eqn:optimization_problem_constrained} using the Picard operator. Then, using the contraction properties of the Picard operator, we eliminate the need for the infinite-dimensional variable $\mbf u$ and propose an algorithm for solving the optimization problem. Finally, the proposed algorithm is extended through the use of extended Picard iterations. 
\subsection{Picard Operator for the Parameter Estimation Problem}\label{subsec:picard_operator}

Let us first introduce the Picard Operator.\vspace{-1mm}
\begin{defn}
\label{def:picard_operator}
For a given $t_0 \in \R$, $T > 0$ and continuous vector field, $f:\Gamma \times\R^{n_x} \times \Theta \rightarrow \R^{n_x}$, we define the Picard operator $\mcl P_{t_0, x, \theta} : C[0,T]\rightarrow  C[0,T] $\vspace{-2mm}
 \begin{equation}
 (\mcl P_{t_0, x, \theta} \mbf u)(t) := x+ \int_{0}^{t} f(s + t_0, \mbf u(s), \theta) ds,\label{eqn:picard_operator_fixed} \vspace{-2mm}
 \end{equation} 
 and when $t_0\hspace{-1mm}=\hspace{-1mm}0$, we simplify the notation as $\mcl P_{x, \theta}\hspace{-1mm}:=\hspace{-1mm}\mcl P_{0, x, \theta}$.
\end{defn}\vspace{-1mm}

The next theorem establishes that for sufficiently small $T$, the Picard Operator is a contraction map and, therefore,  for {any} $\mbf u(t)$, $\mcl P_{x, \theta} \circ \cdots \circ \mcl P_{x, \theta} \mbf u \rightarrow \mbf u^*$ where $\mcl P_{x, \theta} \mbf u^*=\mbf u^*$.\vspace{-1mm}

\begin{thm}[Picard-Lindel\"of Theorem]\label{thm:picard_contraction}
Let $\Gamma \subset \R$, $X\subset\R^{n_x}$, $\Theta\subset\R^{n_{\theta}}$ be compact and $f \in C(\Gamma\times\R^{n_x}  \times\Theta)$ be locally Lipschitz. Define $C_{a, T} \hspace{-0.5mm}=\hspace{-0.5mm} \{\mbf u \in C([0,T])  |  \|\mbf u\|_{\infty} \leq a\}$. Then, for all $t_0 \in \Gamma$, $x \in X$ and $\theta \in \Theta$, there exist $T > 0$, $a > 0$ and $\mbf u^* \in C_{a, T}$ such that   $\lim\limits_{k \rightarrow \infty}{\mcl P}_{t_0, x, \theta}^k \mbf u=\mbf u^*$ for any $\mbf u \in C_{a, T}$.
\end{thm}\vspace{-4mm}
\begin{proof}
See Rectification Thms. in~\cite{arnold1992ordinary}.
\end{proof} \vspace{-2mm}


The Picard operator allows us to restate the optimization Problem in~\eqref{eqn:optimization_problem_constrained} as\vspace{-2mm}
\begin{equation} \label{eqn:optimization_problem_constrained_approximated}
   \min_{\substack{x \in X, \theta \in \Theta, \\ \mbf u \in C(\Gamma)}} \hspace{-2mm}  L(x, \theta, \mcl P_{x, \theta}^n\mbf u)\quad  
     \text{s.t.} \quad    \mcl P_{ x, \theta} \mbf u = \mbf u,\vspace{-2mm}
\end{equation} 
where, to exploit the KKT conditions in~\eqref{eqn:optimality_conditions}, we introduce  $\mcl P_{x, \theta}^n$ in the loss function. Note that $L(x, \theta, \mcl P_{x, \theta}^n\mbf u) = L(x, \theta, \mbf u)$ when $\mbf u$ is a fixed point of $\mcl P_{x,\theta}$ -- i.e. $\mcl P_{x, \theta}^n \mbf u = \mbf u$. 

Although the restatement of Optimization Problem~\eqref{eqn:optimization_problem_constrained} still has an infinite-dimensional variable $\mbf u$, the contractive properties of the Picard operator allow us to propose an algorithm for solving the optimization problem in~\eqref{eqn:optimization_problem_constrained_approximated} as established in the next subsection.

\subsection{Algorithm for solving the Parameter Estimation Problem} 
\label{subsec:algorithm}
\begin{figure}[t]
\begin{algorithm}[H]
\begin{algorithmic}
{\small
\item
\texttt{INPUT:} $N$ -- the number of iterations \\
\hspace{11mm} $n$ -- the order of Picard iterations\\
\hspace{11mm} $x_0 \in X$, $\theta_0 \in \Theta$, ${{\mbf u}}_0 = x_0$, $\alpha  > 0, \sigma \in (0, 1]$\\
\texttt{For} $k$ from $0$ to $N-1$\\
\qquad $\theta_{k+1} := \Pi_\Theta[ \theta_{k} - {\alpha \nabla_\theta L(x_{k}, \theta_{k}, \mcl P^n_{x_{k}, \theta_{k}}{{\mbf u}}_{k}))}]$ \\
\qquad $x_{k+1} :=  \Pi_X[x_{k} - {\alpha \nabla_{x} L(x_{k}, \theta_{k}, \mcl P^n_{x_{k}, \theta_{k}}{{\mbf u}}_{k})}]$ \\
\qquad${{\mbf u}}_{k+1}(t)  := (1-\sigma) {{\mbf u}}_{k}(t) + \sigma(\mcl P_{x_{k+1}, \theta_{k+1}} {{\mbf u}}_{k})(t)$ \\
\texttt{EndFor} \\
\texttt{OUTPUT:} $x_{N}, \theta_{N}, {{\mbf u}}_{N}$
}
\end{algorithmic}
\caption{Gradient-Contract-Multiplier Algorithm}  \label{alg3:modified_picard}
\end{algorithm}\vspace{-8mm}
\end{figure}
In this subsection, we propose an algorithm for solving the optimization problem in~\eqref{eqn:optimization_problem_constrained_approximated}. To achieve the goal, we focus on the gradient-projection algorithm for constrained optimizations. Gradient-projection algorithms update the decision variables by taking a gradient step and then projecting the solution to the feasible set that can be defined as $\{\mbf u\;|\; \mbf u = \mcl P_{x,\theta}\mbf u\}$. However, the infinite-dimensional nature of $\mbf u$ prevents direct application of gradient-projection to problem~\eqref{eqn:optimization_problem_constrained_approximated}. Instead, Alg.~\ref{alg3:modified_picard} uses the contractive properties of the Picard Operator to update the infinite-dimensional variable ${{\mbf u}}_k$ -- ensuring $\mbf u_k = \mcl P_{x,\theta}\mbf u_k$ only in the limit $k \rightarrow \infty$.

 In the first step, given $\mbf u_k$, Alg.~\ref{alg3:modified_picard} applies gradient descent for finite-dimensional variables $x, \theta$ as\vspace{-2mm}
\begin{align} 
\theta_{k+1} &= \Pi_\Theta[\theta_k - {\alpha \nabla_\theta L(x_k, \theta_k, {{\mbf u}}_k)}] \label{eqn:alg_step1}\\ 
 x_{k+1} &= \Pi_X[x_k - {\alpha \nabla_{x} L(x_k, \theta_k, {{\mbf u}}_k)}], \notag\\[-8mm]\notag
\end{align} 
where the gradient is taken for the loss function and the Picard Operator. Here $\alpha > 0$ is a step size and ${\Pi}_X : \R^{n_x} \rightarrow \R^{n_x}$ is a projection operator defined as\vspace{-2mm}
\begin{equation}\label{eqn:projection_operator}
    {\Pi}_X(z) := \underset{x \in X}{\texttt{argmin}} \|z - x\|^2_2.\vspace{-2mm}
\end{equation}
Then, in step 2, Alg.~\ref{alg3:modified_picard} uses the Picard operator  to update the infinite-dimensional variable $\mbf u$ as follows\vspace{-2mm}
\begin{equation}
{{{\mbf u}}_{k+1}} = (1- \sigma) {{\mbf u}}_k + \sigma \mcl P_{ x_{k+1}, \theta_{k+1}} \mbf u_k,\label{eqn:alg_step2}\vspace{-2mm}
\end{equation}
where $\sigma$ is a step size. Note that if $\mcl P_{ x_{k+1}, \theta_{k+1}}$ is a contraction, the second step of Alg.~\ref{alg3:modified_picard} updates $\mbf u$ closer to the feasible set as shown to Thm.~\ref{thm:picard_contraction}. 

To analyze the convergence of Alg.~\ref{alg3:modified_picard}, let us consider a fixed point of Alg.~\ref{alg3:modified_picard}. Specifically, the fixed point $\{x, \theta, \mbf u\}$  satisfies\vspace{-2mm}
\begin{align}
&\nabla_{x,\theta} L(x, \theta, \mcl P^n_{ x, \theta}\mbf u) = \nabla_{x,\theta} L(x, \theta, \mbf u) \label{eqn:KKT_approximation}\\[-1mm]
&\qquad +\sum\nolimits_{i = 1}^{N_s}\nabla_{u(t_i)} L(x, \theta, \mbf u) \nabla_{x,\theta} (\mcl P^n_{ x, \theta}\mbf u)(t_i, x, \theta)=0.\notag \\[-7.5mm]\notag
\end{align}

The conditions in~\eqref{eqn:KKT_approximation} are then an approximated version of the KKT optimality conditions in Eqn.~\eqref{eqn:optimality_conditions} if $(\nabla_{x, \theta}\mcl P^n_{ x, \theta}\mbf u)(t_i, x, \theta) \rightarrow \nabla_{x, \theta} \phi(t_i, x, \theta)$, where $\phi$ is the solution map as in~\eqref{eqn:solutionmap}. The next theorem establishes the convergence of the gradients of the Picard operator. \vspace{-1mm}

\begin{thm}\label{thm:picard_sensitivity}
Let $\Gamma, X, \Theta$ be compact, $f \in C^1(\Gamma\times\R^{n_x}  \times\Theta)$ and $\nabla_x f, \nabla_\theta f$ be locally Lipschitz continuous. Define $C_{a,T} = \{\mbf u \in C([0, T])| \|\mbf u\|_\infty \leq a \}$.
 Then for all $t\in\Gamma$, $x \in X$, $\theta\in \Theta$, there exists $a > 0$ and $T>0$ such that  for any  $\mbf u \in C_{a, T}$ we have $\lim_{n\rightarrow \infty} \nabla_{x,\theta} ({\mcl P}_{ x, \theta}^n \mbf u)(t) = \nabla_{x, \theta} \phi(t , x, \theta)$
where $\phi(t, x, \theta)$ is the solution map (Eqn.~\eqref{eqn:solutionmap}).
\end{thm} \vspace{-3mm}
 \begin{proof}
See Prop.~14 in~\cite{talitckii2024picard}.
\end{proof} \vspace{-2mm}
By Thm.~\ref{thm:picard_sensitivity}, the fixed point of Alg.~\ref{alg3:modified_picard} satisfies the optimality conditions of the constrained optimization problem in~\eqref{eqn:optimization_problem_constrained} in an approximated sense. Numerical examples demonstrate that for practical usage, $n = 1, 2$ or $3$ is sufficient for highly accurate solutions. However, Alg.~\ref{alg3:modified_picard} converges to the fixed point only if the Picard Operator is a contraction, which can be guaranteed by choosing a sufficiently small $T$. Thus, the time interval for which the Picard operator is a contraction may be smaller than the interval on which the data is sampled. 
For such cases, we extend Alg.~\ref{alg3:modified_picard} in the next subsection.

\subsection{Extended Algorithm for Parameter Estimation}\label{subsec:algorithm_extended}
\begin{figure}[t]
\begin{algorithm}[H]
\begin{algorithmic}
{\small
\item
 \texttt{INPUT:} $N$ -- the number of iterations \\
 \hspace{11mm} $n$ -- the order of Picard iterations \\
 \hspace{11mm} $J$ -- the number of time intervals \\
 \hspace{11mm} $T > 0$, $\theta_0 \in \Theta$, $x_0 \in X^{\otimes J}$, ${\mbf u}_{0} \in C^{\otimes J}[0, T)$,  \\
 \hspace{11mm} $\lambda \geq 0$, $\alpha  > 0, \sigma \in (0, 1]$\\
 \texttt{INIT:} $\mbf u_j(t) = x_{0,j}$ for all $j \in \overline{1, J}$\\
 \texttt{DEFINE:} $L_{\lambda, n}(x, \theta,\mbf u)$ as in Eqn.~\eqref{eqn:extended_loss}\\
 \texttt{For} $k$ from $0$ to $N-1$\\
\quad $\theta_{k+1} := \Pi_\Theta[ \theta_{k} - \alpha \nabla_\theta L_{\lambda, n} (x_{k}, \theta_{k},{\mbf u}_{k})]$ \\
\quad \texttt{For} $j$ from $1$ to $J$\\
\quad\quad $ {x}_{k+1, j} := \Pi_X[x_{k, j} - {\alpha \nabla_{x_j} L_{\lambda, n}(x_k, \theta_k, {\mbf u}_k \big)}]$ \\ 
\quad\quad $\mbf u_{k+1, j}  := (1-\sigma) \mbf u_{k, j} + \sigma(\mcl P_{(j-1)T, x_{k+1, j}, \theta_{k+1}} \mbf u_{k, j}) $ \\
\qquad \texttt{EndFor}\\ 
\texttt{EndFor} \\
\texttt{OUTPUT:} $x_{N}, \theta_{N}, {{\mbf u}}_{N}$
}
\end{algorithmic}
\caption{Extended Grad-Contract-Multiplier Alg.}  \label{alg4:ext_picard}
\end{algorithm}\vspace{-8mm}
\end{figure}

To apply Alg.~\ref{alg3:modified_picard} for data sampled on large time intervals, we consider a modified version of extended Picard iterations~\cite{peet2010converse}. Specifically, we create a partition of the time domain into smaller disjoint subintervals, such that the Picard Operator is contractive on each of these subintervals. Formally, given a convergence interval $T>0$, we assume that $\Gamma \subseteq [0, JT)$. Then, we lift the the solution  $\mbf u \in C[0, T]^{\otimes J}$ and initial conditions $x \in X^{\otimes J} = X \times \dots \times X$. The solutions $\mbf u_j$ represent estimated solutions on the time interval $t \in [(j-1)T, jT)$. Then, the optimization problem in~\eqref{eqn:optimization_problem_constrained} may be reformulated as \vspace{-2mm}
\begin{align} 
\hspace{-1mm}\min_{\substack{x \in X^{\otimes J},\, \theta \in \Theta\\ \mbf u \in C[0,T]^{\otimes J}}} &\;   \frac{1}{2N_s} \hspace{-1mm} \sum\nolimits_{i=1}^{N_s}\|y_i - g(t_i, \mbf u_{\lfloor t_i/T \rfloor}(t_i\hspace{-0mm}\bmod{T}), \theta)\|_2^2\notag\\[-2mm]
     \text{s.t.} \; & \mcl P_{(j-1)T, x_j, \theta} \mbf u_j = \mbf u_j \qquad \forall j ,\,  t\in [0,T] \label{eqn:optimization_problem_constrained_extended1} \\
      & \mbf u_{j}(T)  = x_{j+1} \qquad \forall j \in \overline{1, J-1},\notag\\[-8mm]\notag 
\end{align}
Constraints in the optimization problem~\eqref{eqn:optimization_problem_constrained_extended1} ensure that the optimization problems in~\eqref{eqn:optimization_problem_constrained} and~\eqref{eqn:optimization_problem_constrained_extended1} are equivalent.\vspace{-1mm}
\begin{thm}Let $\Gamma = [0, JT]$, then Optimization Problems~\eqref{eqn:optimization_problem_constrained_approximated} and~\eqref{eqn:optimization_problem_constrained_extended1} are equivalent.
\end{thm}  \vspace{-3mm}
 \begin{proof}
See Lem. 4 in~\cite{talitckii2024picard}.
\end{proof} \vspace{-2mm}

The optimization in~\eqref{eqn:optimization_problem_constrained_extended1} has binding constraints $\mbf u_j(T) = x_{j+1}$ and functional equality constraints $\mbf u_j = \mcl P_{(j-1)T, x_j, \theta} \mbf u_j$. Similar to Alg.~\ref{alg3:modified_picard}, we may process the functional constraints using the Picard operator. However, the additional challenge is to consider the binding constraints. To process the binding constraints $\mbf u_j(T) = x_{j+1}$, we introduce a penalty function in the objective function as follows \vspace{-2mm}
\begin{align} \label{eqn:optimization_problem_constrained_extended}
 \hspace{-1mm}  \min_{\substack{x \in X^{\otimes J}, \theta \in \Theta\\ \mbf u \in C[0,T]^{\otimes J}}}\hspace{-0.5mm} &\; \frac{1}{2N_s} \hspace{-1mm} \sum\nolimits_{i=1}^{N_s} \hspace{-0.5mm} \|y_i\hspace{-0.5mm} -\hspace{-0.5mm} g(t_i, \mbf u_{\lfloor t_i/T \rfloor}(t_i\hspace{-0mm}\bmod{T}), \theta)\|_2^2 \notag \\[-5mm]
    &\qquad\qquad + \lambda \sum\nolimits_{j=1}^{J-1}\|\mbf u_{j}(T) - x_{j+1}\|_2^2\\[-1mm]
     \text{s.t.} \;&\;  \mcl P_{(j-1)T, x_j, \theta} \mbf u_j  = \mbf u_j  \qquad \forall j \in \overline{1, J},\notag\\[-7mm]\notag
\end{align}
where $\lambda \geq 0$ is a regularization parameter. 

Next, since $\mcl P_{(j-1)T, x_j, \theta} \mbf u_j  = \mbf u_j$, the objective function in the optimization in~\eqref{eqn:optimization_problem_constrained_extended} may be reformulated as \vspace{-2mm}
\begin{align}\label{eqn:extended_loss}
& L_{\lambda, n}(x, \theta, \mbf u) \hspace{-1mm}:= \hspace{-1mm} \sum_{j=1}^{J-1} \lambda \|\mcl P^n_{(j-1)T, x_j, \theta} \mbf u_{j}(T) - x_{j+1}\|_2^2\\[-3mm]
&\hspace{-3mm} +\hspace{-1mm}\frac{1}{2N_s}\hspace{-1mm}  \sum_{i=1}^{N_s}\hspace{-1mm}\|y_i \hspace{-0.5mm}-\hspace{-0.5mm} g(t_i, (\mcl P^n_{(j-1)T, x_j, \theta} \mbf u_{\lfloor t_i/T \rfloor})(t_i\hspace{-0mm}\bmod{T}), \theta)\|_2^2. \notag  \\[-8mm]\notag
\end{align} 

For the function $L_{\lambda, n}$ we define Alg.~\ref{alg4:ext_picard}.

Although the extended algorithm (Alg.~\ref{alg4:ext_picard}) allows us to approximately solve the parameter estimation problem on an arbitrary time interval, the approach introduces additional challenges. Specifically, replacing the binding constraints $\mbf u_j(T) = x_{j+1}$ with the penalty function may lead to discontinuities in the estimated solution $\mbf u$. The parameter $\lambda$ is a weight in the objective function, which allows us to increase the penalty function -- implying smaller discontinuities. However, the increased $\lambda$ also reduces the sensitivity to the data samples and hence implies the need for smaller step sizes to ensure convergence. In the next sections, we consider a parameter estimation problem for two attitude dynamics models: the F/A-18 aircraft dynamical model and the 4 control-moment gyro (CMG) model of spacecraft motions.

%

\section{Parameter Estimation of Rigid-Body Dynamics with Single Gimbal Control Moment Gyros.}\label{sec:control-moment-gyros}
In this section, we apply Alg.~\ref{alg4:ext_picard} to estimate the inertia tensor of a spacecraft. Specifically, we consider a 4 control-moment gyro (CMG) model of spacecraft motion, as derived in~\cite{schaub1998feedback}.
%
The model has states of the form\vspace{-2mm}
\begin{equation} \label{eqn:states-CMG}
x(t) = [\omega^T(t), \gamma^T(t), \dot\gamma^T(t), \Omega^T(t)]^T \in \R^{15},\vspace{-2mm}
\end{equation}
where $\omega(t) \in \R^3$ is the satellite angular velocity in the body-fixed coordinates, $\gamma(t) \in \R^4$ are gimbal angles and $\Omega(t) \in \R^4$ are angular rates of the gyro's wheels. 

We assume that the only measurable outputs of the system are $\omega(t)$ and $\gamma(t)$. The inputs $M_g(t) \in \R^4$ and $M_h(t) \in \R^{4}$ are the gimbal and gyro's wheel motor torques.\vspace{-1.5mm}
\begin{equation} \label{eqn:outputs-inputs-CMG}
y(t) = g(x(t)) = [\omega^T(t), \gamma^T(t)]^T\hspace{-2mm}, \;\; u(t) = [M_g(t)^T, M_h(t)]^T\hspace{-2mm}. \vspace{-1mm}
\end{equation}
We estimate the inertia tensor of the spacecraft body, $J_B$. Symmetry of the spacecraft implies that the inertia tensor has a diagonal structure --- i.e., $J_B \hspace{-0.5mm}=\hspace{-0.5mm}\mathbf{diag}(J_{Bx},J_{By},J_{Bz})$. Thus, estimation parameters are $\theta\hspace{-0.5mm} =\hspace{-0.5mm} [J_{B_x}, J_{B_y}, J_{B_z}]\hspace{-0.5mm} \in \hspace{-0.5mm}\R^3$.

\subsection{The Equations of Motions}

Let the spacecraft reference frame be given by $\{b_x, b_y, b_z\}$ and the gimbal reference frames be defined by $\{f_i, g_i, h_i\}$, where $h_i$ represents the main gyroscope axis, $g_i$ is a fixed gimbal axis in the spacecraft frame, and $f_i$ defines the gyroscopic moment output. Then, given gimbal angles $\gamma(t)$, the rotated gimbal frames have the form \vspace{-2mm}
\begin{align*}
h_i(t) &= h_i(0) \cos\gamma(t) + f_i(0) \sin\gamma(t)\\
f_i(t) &= -h_i(0) \sin\gamma(t) + f_i(0) \cos\gamma(t).\\[-8mm]
\end{align*}
Inertia tensors for gimbal frame and the gyroscopic wheel are diagonal matrices of the form $J_G = \mathbf{diag}(J_{Gx},J_{Gy},J_{Gz})$, $J_W = \mathbf{diag}(J_{Wx},J_{Wy},J_{Wz})$.

Then, the total system inertia of the system has the form\vspace{-2mm}
\[
J_S(t) = J_B + J_{Tx} F(t) F(t)^T + J_{Ty } G G^T + J_{Tz} H(t) H(t)^T,\vspace{-2mm}
\]
where $J_{Ti} = J_{Gi} + J_{Wi}$, $G = [g_1, \dots, g_n]$ and\vspace{-2mm}
\begin{equation*}
F(t) = [f_1(t), \dots, f_n(t)],\;\; H(t) = [h_1(t), \dots, h_n(t)].\vspace{-2mm}
\end{equation*}
Next, the equations of motions are established in~\cite{valk2018directional}. \vspace{-2mm}
\begin{align*}
 \Phi(t)\bmat{\dot \omega(t) \\ \ddot \gamma(t) \\ \dot \Omega(t)} \hspace{-1mm}=\hspace{-1mm}   \hspace{-0.5mm}\bmat{\mbf 0 \\ M_g(t) \\ M_h(t)}\hspace{-0.5mm} - \hspace{-0.5mm}\bmat{M_\omega(t)\hspace{-0.5mm} +\hspace{-0.5mm} \omega(t)\hspace{-1mm} \times \hspace{-1mm}(J_s(t) \omega(t)) \\ -M_\gamma(t)\\ -M_\Omega(t)},\hspace{-0.5mm} \\[-7.5mm]
\end{align*}
where $M_g(t), M_h(t) \in \R^4$ are the control inputs representing motor torques and 
$\Phi(t)$ is defined as \vspace{-2mm}
\[\small
\Phi(t) = \bmat{J_s(t) & J_{Ty} G & J_{Wz} H(t) \\ J_{Ty} G^T & J_{Ty}\mbf I & \mbf 0 \\ J_{Wz} H(t)^T & \mbf 0 & J_{Wz} \mbf I}.\vspace{-2mm}
 \] 
 We use $M_\omega(t)$, $M_\gamma(t)$  for notational clarity to denote\vspace{-2mm}
\begin{align*}
M_\omega & = J_{Wz} F \mathbf{diag}(\Omega)(\gamma + G^T \omega)- J_{Wz} G\mathbf{diag}(\Omega) F^T\omega\\[-0.5mm]
& \qquad+ (J_{Tz} - J_{Tx} + J_{Ty})H \mathbf{diag}(F^T \omega) \dot \gamma \\[-0.5mm]
& \qquad+ (J_{Tz} - J_{Tx} - J_{Ty})F\mathbf{diag}(H^T \omega) \dot\gamma\\ 
M_\gamma &=\mathbf{diag}(F^T \omega) ((J_{Tz} - J_{Tx}) H^T \omega + J_{Wz} \Omega)\\
M_\Omega &= -J_{Wz} \mathbf{diag}(F^T \omega)\dot\gamma.\\[-7mm]
\end{align*}

\subsection{Estimation of the Inertia Tensor of the Rigid-Body}

%


To estimate the unknown parameters $\theta \in \R^3$, we assume that the inertia tensor for the gyro's wheel is $J_{Wx} = J_{Wy} = 0.5 J_{Wz} =  1\; \text{kg}\cdot\text{m}^2$. The inertia tensor for the gimbal frame  is $J_{Gx} = J_{Gy} = J_{Gz} = 0.01\; \text{kg}\cdot\text{m}^2$. Finally, we consider four possible inputs for the system of the form $u_i(t) = [M_{g, i}(t)^T, M_{h, i}(t)^T]$, where $M_{g,1}(t) = M_{h, 1}(t) = 0$; $M_{g,2}(t) = 40cos(10t)$ and $M_{h, 2}(t) = 0$; $M_{g,3}(t) = 0$ and $M_{h, 3}(t) = 100cos(10t)$; $M_{g,4}(t) =\begin{cases} -4 & \text{if } t < 1.5 \\ 0 & \text{otherwise} \end{cases}$ and $M_{h, 4}(t) = 0$.

For each input,  $j = 1, ..., 4$ we generate a data set of the form\vspace{-2mm}
\[
y_{ij} =  (1 + n_i) g(\phi_{u_j}(t_i, x_0, \theta)) + m_i,\vspace{-2mm}
\]
where $g$ is the output function as in~\eqref{eqn:outputs-inputs-CMG} and $\phi_{u_j}$ is the solution map representing the dynamic of the states in~\eqref{eqn:states-CMG} with input $u_j$. $n_i, m_i$ are normally distributed random variables with zero mean and variances of $0.05$. True values of the inertia tensor of the spacecraft (estimated parameters $\theta$) are $J_B = \textbf{diag}([214, 201, 500])$, the time samples are $t_i =  \{0:0.01:3\}$, and initial conditions are given as\vspace{-2mm}
\begin{align*}
 &\gamma_0 = [0.1, 0.1,0.1,0.1]^T, \omega_0 = [0.1,0.1,0.1]^T, \\
 &\dot \gamma_0 = [0.1,0.1,0.1,0.1], \Omega_0 = [10,10,10,10]^T.\\[-8mm]
\end{align*}

For estimating parameters, we use Alg.~\ref{alg4:ext_picard} with $\lambda = 10$, $n = 3$ and $T = 0.2$. Alg.~\ref{alg4:ext_picard} returns estimates of the parameters is $\hat \theta = [215, 207, 527]$. The simulation results for estimated parameters and sampled data are presented in Fig.~\ref{fig:CMG}.

\begin{figure}
\centering\vspace{1mm}
\includegraphics[width = 0.9\linewidth]{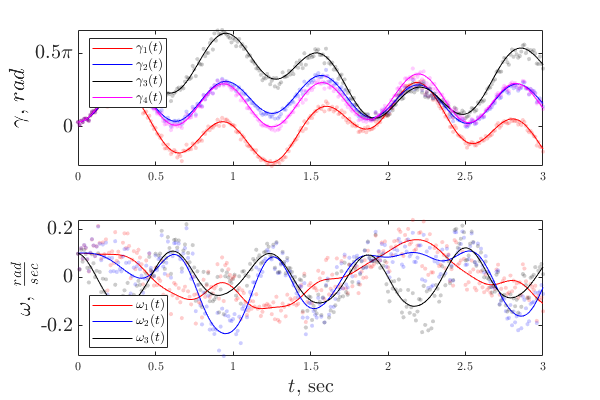}\vspace{-2mm}
\caption{Parameter Estimation using Alg.~\ref{alg4:ext_picard} for the Rigid-Body Satellite with $4$ SGCMGs. The dots represent measurements, the solid line is the predicted outputs for input $u_2$. True parameter values are
$\theta = [214, 201, 500]$. Estimated parameter values are $\hat \theta = [215, 207, 527]$.}\vspace{-5mm}
\label{fig:CMG}
\end{figure}

%
 
\section{Parameter Estimation in Flight Dynamics}\label{sec:FA18-model}

In this section, we consider a model of aircraft attitude dynamics. Specifically, we estimate control surface coefficients in the model of the F/A-18 aircraft. The dynamics describe the position, orientation, and velocity of the vehicle. Although the basics of the aircraft dynamics are often described from first principles, the complexity of aircraft geometry and aerodynamics implies the need for experimental data to estimate model parameters such as lift coefficients~\cite{song2002comparative}. 

The position of the aircraft can be described by the following states\vspace{-2mm}
\begin{equation}\label{eqn:states-FA}
x(t)  =[ v(t), \alpha(t), \beta(t), p(t), q(t), r(t), \theta(t), \phi(t), \psi(t)] \in \R^9,\vspace{-1mm}
\end{equation}
where $v(t)$ represents the aircraft velocity, $\alpha(t)$ is the angle-of-attack, $\beta(t)$ is the side-slip angle, $p(t)$ is the roll rate, $q(t)$ is the pitch rate, $r(t)$ is the yaw rate, $\theta(t), \phi(t), \psi(t)$ are the pitch, bank and yaw angles. 

The inputs of the system have the form \vspace{-2mm}
\begin{equation}\label{eqn:inputs-FA}
u(t) = [T(t), \delta_{stab}(t), \delta_{rud}(t), \delta_{ail}(t)] \in \R^4,\vspace{-2mm}
\end{equation}
where $T(t)$ is the thrust,  $\delta_{stab}(t)$, $\delta_{rud}(t)$, $\delta_{ail}(t)$ are the control surface angles for vertical stabilizers, rudders, and ailerons.
The dynamics of the states are defined by forces and angular moments acting on the aircraft. Specifically, the drag $F_D$, side $F_S$ and lift $F_L$ forces have the form\vspace{-1mm}
\begin{equation*}
\bmat{F_D(t) \\ F_S(t)\\ F_L(t)} \hspace{-0.5mm}= \hspace{-0.5mm}\bmat{\bar F(t)C_D(x(t), u(t))\\ \bar F(t)C_S(x(t), u(t))\\\bar F(t)C_L(x(t), u(t))},
\end{equation*}
where  $\bar F(t) = 0.5\rho v(t)^2S$,  and where $\rho = 0.075\, \text{lb}\cdot \text{ft}^{-3}$ is the air density and $S = 400\, \text{ft}^2$ is the wing area. $C_D, C_S, C_L$ are functions representing aerodynamic coefficients, which typically depend on the states and inputs of the system.

The roll, pitch, and yaw moments are given by\vspace{-2mm}
\begin{align*}
\bmat{M_r(t) \\ M_p(t)\\ M_y(t)} \hspace{-0.5mm}=   \hspace{-0.5mm}\bmat{\bar M_1(t)C_r(x(t), u(t))\\\bar M_2(t)C_p(x(t), u(t))\\\bar M_2(t)C_y(x(t), u(t))},\\[-7mm]
\end{align*}
where $\bar M_1(t)\hspace{-0.5mm} =\hspace{-0.5mm} 0.5 \rho v(t)^2 S \bar c$, $\bar M_{2}(t)\hspace{-0.5mm}=\hspace{-0.5mm}0.5 \rho v(t)^2 S b $,  $c = 11.45\, \text{ft}$ is the mean aerodynamic chord and $\bar b =  37.42\, \text{ft}$ is the wing span. $C_r, C_p, C_y$ are aerodynamic coefficients of roll, pitching, and yaw moments.

For aerodynamics coefficients, we consider the model derived in~\cite{chakraborty2011susceptibility}\vspace{-2mm}
\begin{align*}
C_D &= \mu_{D,0} +\mu_{D_{\alpha}}^T Z_4(\alpha) \cos(\beta) +  \mu_{D_{stab}}^T Z_3(\alpha) \delta_{stab}\\
C_L &= \mu_{L_{\alpha}}^T Z_3(\alpha)\cos(\frac{2\beta}{3}) + \mu_{L_{stab}} Z_3(\alpha)\delta_{stab} \\
C_S &= \mu_{S_{\beta}}^T Z_2(\alpha)\beta  + \mu_{S_{{ail}}}^T Z_3(\alpha)\delta_{ail}  + \mu_{S_{{rud}}} Z_3(\alpha)\delta_{rud}\\
C_r &=  \mu_{r_{\alpha}}^T Z_4(\alpha)\beta + \mu_{r_{{ail}}} Z_3(\alpha)\delta_{ail} + \mu_{r_{{rud}}}^T Z_3(\alpha)\delta_{rud} \\
&\qquad\qquad+  \mu_{r_{p}}^T Z_1(\alpha) \frac{bp}{2v}+ \mu_{r_{r}}^T Z_1(\alpha)\frac{br}{2v}\\[-1mm]
C_p &= \mu_{p_{\alpha}}^T Z_2(\alpha) +\mu_{p_{stab}}^T Z_2(\alpha) \delta_{stab} + \mu_{p_{q}}^T Z_3(\alpha) \frac{\bar c q}{2v}\\[-1mm]
C_y &= \mu_{y_{\beta}}^T Z_2(\alpha)\beta  + \mu_{y_{rud}}^T Z_4(\alpha)\delta_{rud} + \mu_{y_{ail }}^T Z_3(\alpha) \delta_{ail}\\
&\qquad\qquad + \mu_{y_{p}}^T Z_1(\alpha)\frac{bp}{2v} + \mu_{y_{r}} ^TZ_1(\alpha)\frac{br}{2v},\\[-8mm]
\end{align*} 
where $Z_m(\alpha)$ is the vector of monomials with degree less or equal $m$ and $\mu_{i}$ are vectors of parameters with appropriate dimensions. 
The values of the coefficients are given as in Tables 2 and 3 in~\cite{chakraborty2011susceptibility}. 

We assume that coefficients for drag force $\mu_{D,0}, \mu_{D, \alpha}, \mu_{D_{stab}}$ and the roll moment $\mu_{r_\alpha}, \mu_{r_p}, \mu_{r_r}, \mu_{r_{ail}}, \mu_{r_{rud}}$ are unknown parameters in the model. In this section, we estimate these 28 unknown parameters -- i.e., $\theta$ has the form\vspace{-2mm}
\begin{equation}\label{eqn:parameters-FA}
\theta  = [\mu_{D,0}^T, \mu_{D, \alpha}^T, \mu_{D_{stab}}^T, \mu_{r_\alpha}^T, \mu_{r_p}^T, \mu_{r_r}^T, \mu_{r_{ail}}^T, \mu_{r_{rud}}^T] \in \R^{28}.\vspace{-3mm}
\end{equation} 
%
\subsection{The Equations of Motions}

The attitude dynamics of the F/A-18 are modelled as\vspace{-2mm}
\begin{align*}
     &m \dot v(t) \hspace{-0.5mm}= \hspace{-0.5mm} -\hspace{-0.5mm} F_S(t) \sin\beta(t)\hspace{-0.5mm} + \cos\beta(t)(- F_D(t) \hspace{-0.5mm}+\hspace{-0.5mm} T(t)\cos\alpha(t)) \\
    &\quad+ \hspace{-0.5mm}mg\sin\phi(t) \cos\theta(t) \sin\beta(t)\hspace{-0.5mm}-\hspace{-0.5mm} mg\sin\theta(t) \cos\alpha(t)\cos\beta(t)  \\ 
    &\quad+\hspace{-0.5mm} mg\cos\phi(t)\cos\theta(t) \sin\alpha(t) \cos\beta(t) \\
   &mv(t)\cos\beta(t) \dot \alpha(t) = -F_L(t) + mv(t) cos\beta(t) q (t)\\
   &\quad-mv\sin\beta(t)(p(t)\cos\alpha(t) + r(t) \sin\alpha(t)) - T(t)\sin\alpha(t)   \\
   &\quad+ mg \big(\cos\phi(t)\cos\theta(t) \cos\alpha(t)+ \sin\alpha(t)\sin\theta(t)\big)\\
    &mv(t)\dot \beta(t) \hspace{-0.5mm}= \hspace{-0.5mm}F_S(t)\cos\beta(t)\hspace{-0.5mm} + \hspace{-0.5mm}\sin\beta(t)(F_D(t) \hspace{-0.5mm}+\hspace{-0.5mm} T(t) \cos\alpha(t)) \\
    &\quad- mv(t)r(t)\cos \alpha(t) + mg\cos\beta(t)\sin\phi(t)\cos\theta(t)\\
    &\quad+ mg\sin\beta(t)\cos\alpha(t) \sin\theta(t) + mv(t)p(t) \sin\alpha(t)  \\
    &\qquad- mg \sin(\beta)\sin\alpha(t) \cos\phi(t) \cos\theta(t), \\[-7mm]
\end{align*}
where $g = 32.17 \, \text{ft}\cdot\text{s}^{-2}$ is the gravitational constant and $m = 1034\, \text{slug}$ is the mass of the aircraft. 

Euler angles can be determined as\vspace{-2mm}
\begin{align*}
\bmat{\dot\phi(t) \\ \dot \theta(t) \\ \dot \psi(t)}  &= \bmat{1 & \sin\phi(t) {\tan\theta(t)} & \cos\phi(t){\tan\theta(t)} \\ 0  &\cos\phi(t) &-\sin\phi(t) \\ 0& \sin\phi(t){\sec\theta(t)} & \cos\phi(t) {\sec\theta(t)}}  \hspace{-1mm}\bmat{p(t) \\ q(t) \\ r(t)}\\[-1mm]
    \bmat{\dot p(t) \\ \dot q(t) \\ \dot r(t)} &= \mbf J \Biggl(\bmat{M_r(t) \\ M_p(t)\\ M_y(t)} \hspace{-1mm} - \hspace{-1mm}\bmat{p(t) \\ q(t) \\ r(t)} \hspace{-1mm}\times \hspace{-1mm} \bmat{I_{xx} p(t) - I_{xz} r(t) \\ I_{yy} q(t) \\ -I_{xz}p(t) + I_{zz}r(t)} \Biggr), \\[-7mm]
\end{align*}
where $I_{xx} = 2.3\cdot 10^4\,\text{slug}\cdot \text{ft}^2, I_{yy} =  151\,293\,\text{slug}\cdot \text{ft}^2, I_{zz} = 169\,945\,\text{slug}\cdot \text{ft}^2$ are roll, pitch and yaw axis moments of intertia, $I_{xz} = -2971\,\text{slug}\cdot \text{ft}^2$ is the cross product of inertia about $y$ axis and where \vspace{-3mm}
$${\small
\mbf J = \frac{1}{I_{xx}I_{zz} - I_{xz}^2} \bmat{I_{zz} & 0 & I_{xz} \\ 0& (I_{xx}I_{zz} - I_{xz}^2)I_{yy}^{-1} & 0 \\ I_{xz} & 0 & I_{xx}}}.\vspace{-2mm}
$$
\subsection{Estimation of the Aerodynamic Coefficients}

\begin{table}[t]\centering
\begin{tabular}{cc|cc|cc}
\multicolumn{2}{c|}{$\mu_{D,0}, \mu_{D,\alpha}$} & \multicolumn{2}{c|}{$\mu_{r_\alpha}$} & \multicolumn{2}{c}{$\mu_{r_p}, \mu_{r_r}$}\\
True\hspace{-2.5mm}  & \hspace{-2.5mm}Estimated\hspace{-2.5mm}  & True\hspace{-2.5mm}  &\hspace{-2.5mm}Estimated\hspace{-2.5mm} &True\hspace{-2.5mm}  & \hspace{-2.5mm}Estimated\hspace{-2.5mm}  \\
 \hspace{-2.5mm}$\begin{array}{r}  1.504 \\ -1.499\\ -0.200\\  6.397\\ -5.734\\  1.461\end{array}$\hspace{-2.5mm}
 & 
 \hspace{-2.5mm}$\begin{array}{r}  1.441 \\ -1.502\\ -0.0552\\  6.507\\ -5.687\\  1.441\end{array}$\hspace{-2.5mm} 
 &
 \hspace{-2.5mm} $\begin{array}{r} -0.056 \\ -0.415 \\ -0.362 \\  2.384 \\ -1.620\end{array}$\hspace{-2.5mm}
 &
 \hspace{-2.5mm}$\begin{array}{r} -0.071 \\ -0.414 \\ -0.316\\  2.376 \\ -1.750\end{array}$\hspace{-2.5mm} 
 &
  \hspace{-2.5mm}$\begin{array}{r}-0.354 \\ 0.238 \\  0.198 \\  0.780 \\ -1.087\end{array}$\hspace{-2.5mm}
  &
  \hspace{-2.5mm} $\begin{array}{r} -0.412 \\ 0.206 \\  0.301 \\  0.951 \\ -0.852\end{array}$ \hspace{-2.5mm}\\
 \hline
\multicolumn{2}{c|}{$\mu_{r_{ail}}$} & \multicolumn{2}{c|}{$\mu_{D_{stab}}$} & \multicolumn{2}{c}{$\mu_{r_{rud}}$}\\
True\hspace{-2.5mm}  & \hspace{-2.5mm}Estimated\hspace{-2.5mm}  & True\hspace{-2.5mm}  &\hspace{-2.5mm}Estimated\hspace{-2.5mm} &True\hspace{-2.5mm}  & \hspace{-2.5mm}Estimated\hspace{-2.5mm}  \\
\hspace{-2.5mm}$\begin{array}{r} 0.142 \\ -0.052\\ -0.265\\ 0.199\end{array}$\hspace{-2.5mm}
&
\hspace{-2.5mm} $\begin{array}{r} 0.141 \\ -0.067\\ -0.236\\ 0.197\end{array}$\hspace{-2.5mm} 
  &
  \hspace{-2.5mm} $\begin{array}{r} 0.037 \\ -0.274\\ 4.236\\ -3.858\end{array}$\hspace{-2.5mm}
  &
  \hspace{-2.5mm} $\begin{array}{r} 0.393 \\ 0.078\\  4.558 \\ -3.579\end{array}$ \hspace{-2.5mm} 
  & 
  \hspace{-2.5mm} $\begin{array}{r} 0.013\\ 0.001 \\ 0.008 \\ -0.027\end{array}$\hspace{-2.5mm}
  &
  \hspace{-4mm} $\begin{array}{r}  0.016\\ -0.040 \\  0.007\\  0.018\end{array}$\hspace{-2.5mm} 
\end{tabular}\vspace{-1mm}
\caption{Parameter estimation results for the F/A-18 model using Alg.~\ref{alg4:ext_picard} with $\lambda = 10$, $n = 3$ and $T = 0.1$. The left columns represent the true values of the estimated parameters. The right columns are the estimated values.}\vspace{-4mm}
\label{tab:results_FA}
\end{table}

To estimate  the unknown parameters in~\eqref{eqn:parameters-FA}, we run $6$ experiments for identification of $\theta_1 = [\mu_{D,0}, \mu_{D,\alpha}]$, $\theta_2 = \mu_{D_{stab}}$, $\theta_3 = [\mu_{r_\alpha}]$, $\theta_4 = [\mu_{r_p}, \mu_{r_r}]$, $\theta_5 = \mu_{r_{ail}}$ and $\theta_6 = \mu_{r_{rud}}$. For each experiment,  we assume that all other parameters are known. Then, we generate the data set of  the form 
\vspace{-2mm}
\[
y_{ij} =  (1 + n_i) \phi_{u_j}(t_i, x_{0j}, \theta)  + m_i,\vspace{-2mm}
\]
where $\phi_{u_j}$ is the solution map representing the dynamic of the states in~\eqref{eqn:states-FA} with the input $u_j$ and the initial condition $x_{0j}$, which are specified later. $n_i$ and $m_i$ are normally distributed random variables with zero mean and a variance of $0.05$. Time points $t_i$ are $t_i \in \{0:0.1:5\}$.

For estimating $\theta_1, \theta_3, \theta_4$, we use the data set generated from two instances (two initial conditions) with inputs of $T(t) = 14500\;\text{lbf}$, $\delta_{stab}(t) = \delta_{ail}(t) = \delta_{rud}(t) = 0$. The initial conditions are\vspace{-2mm}
\begin{align} \label{eqn:initial_condition_testing_FA}
x_{01} &= [450, -1.25, 1.0, 0.175,0,0.0875, 0,0.175,0]^T\\
x_{02} &= [350, 0.35, 0.7, 0.175,1.5,0.0875, 0,0.175,0]^T. \notag\\[-8mm]\notag
\end{align}

For estimating $\theta_5$, we also use the measurements that are taken over two instances with inputs of $T(t) = 14500\;\text{lbf}$, $\delta_{stab}(t) = \delta_{rud}(t) = 0$ and $\delta_{ail}(t) = 0.25I(t \in [0, 2]) - 0.25I(t \in (2, 4])$, where $I(t \in [a, b]) = \begin{cases} 1 & \text{ if } t \in [a, b] \\ 0& \text{ otherwise} \end{cases}$.\vspace{-2mm}\\
 The initial conditions are as in~\eqref{eqn:initial_condition_testing_FA}.

The estimations of $\theta_2$ and $\theta_6$ are more challenging. So in this case, we use the data set for three instances. Two initial conditions are chosen as in~\eqref{eqn:initial_condition_testing_FA} and one additional initial condition is $x_{03} = [350, 1, 1, 0.175,0,0.0875, 0,0.175,0]^T$.

Estimation $\theta_2$, we use inputs of $T(t) = 14500\;\text{lbf}$, $\delta_{ail}(t) = \delta_{rud}(t) = 0$ and $\delta_{stab}(t) =0.25I(t \in [0, 2]) - 0.25I(t \in (2, 4])$. For estimation $\theta_6$, we use input of $T(t) = 14500\;\text{lbf}$, $\delta_{stab}(t) = \delta_{ail}(t) = 0$ and $\delta_{rud}(t) = 0.2I(t \in [0, 2]) - 0.2I(t \in (2, 4])$. The results of our parameter estimation task are presented in Tab.~\ref{tab:results_FA}.

\section{Conclusion}

In this paper, we formulate the grey-box parameter estimation problem as optimization of parameters and solution subject to the constraint that the solution be a fixed point of the Picard operator. To solve this constrained optimization problem, we propose a gradient contractive algorithm, which uses the contractivity of the Picard operator in lieu of a projection to ensure the solution converges to the solution map of the associated nonlinear ODEs. The proposed method allows for sparse, irregular, and includes only partial measurements of the state. Furthermore, the approach does not require numerical simulation or measurements of time derivatives. The method is tested on models of spacecraft and aircraft motion. In the first example, we estimate the inertia tensor in the spacecraft model with $4$ control-moment gyros. In the second example, we estimate control surface coefficients in the model of the F/A-18 aircraft.
 



\bibliographystyle{IEEEtran}
\bibliography{Talitckii_Picard}

\end{document}